\documentclass[a4paper,12pt]{amsart}
\usepackage{fullpage}

\usepackage{hyperref}
\hypersetup{
    colorlinks,
    citecolor=green,
    filecolor=black,
    linkcolor=blue,
    urlcolor=blue
}
\hypersetup{linktocpage}

\usepackage{cite}
\usepackage{graphicx}

\usepackage{amsbsy}
\usepackage{latexsym}
\usepackage{amsfonts}
\usepackage{amssymb}
\usepackage[usenames]{color}
\usepackage{amsmath,amsthm}
\usepackage{enumerate}
\usepackage{tikz}
\usetikzlibrary{arrows}
\usepackage{mathdots}
\usepackage{mathtools}

\usepackage{stmaryrd}
\usepackage{dsfont}
\usepackage{color}

\DeclareMathOperator*{\esssup}{ess\,sup}

\providecommand{\abs}[1]{\lvert#1\rvert}

\providecommand{\norm}[1]{\lVert#1\rVert}

\newtheorem{theorem}{Theorem}
\newtheorem{lemma}[theorem]{Lemma}
\newtheorem{prop}[theorem]{Proposition}
\newtheorem{cor}[theorem]{Corollary}

\theoremstyle{definition}

\theoremstyle{remark}
\newtheorem{remark}[theorem]{Remark}

\newcommand{\cA}{{\mathcal{A}}}

\newcommand{\sdd}{\,\mathrm{d}}

\newcommand{\Chi}{\raise .3ex
\hbox{\large $\chi$}}

\newcommand{\R}{\mathbb{R}}

\DeclareMathOperator{\divergence}{div}

\usepackage[section]{algorithm}
\usepackage{algorithmicx}
\usepackage{algpseudocode}
\algrenewcommand\algorithmicrequire{\makebox[46pt][l]{\textrm{required:}}}
\algrenewcommand\algorithmicensure{\makebox[46pt][l]{\textrm{output:}}}
\algrenewcommand\algorithmicfunction{\textrm{function}}
\algrenewcommand\algorithmicwhile{\textrm{while}}
\algrenewcommand\algorithmicdo{}
\algrenewcommand\algorithmicend{\textrm{end}}
\algrenewcommand\algorithmicforall{\textrm{for all}}
\algrenewcommand\algorithmicfor{\textrm{for}}
\algrenewcommand\algorithmicrepeat{\textrm{repeat}}
\algrenewcommand\algorithmicuntil{\textrm{until}}
\algrenewcommand\algorithmicif{\textrm{if}}
\algrenewcommand\algorithmicthen{\textrm{then}}
\algrenewcommand\algorithmicelse{\textrm{else}}

\usepackage{todonotes}

\newcommand{\be}{\begin{equation}}
\newcommand{\ee}{\end{equation}}
\newcommand{\beq}{\begin{eqnarray}}
\newcommand{\beqq}{\begin{eqnarray*}}
\newcommand{\eeq}{\end{eqnarray}}
\newcommand{\eeqq}{\end{eqnarray*}}

\numberwithin{equation}{section}

\title{Identifiability of Diffusion Coefficients for Source Terms of Non-Uniform Sign}
\author{Markus Bachmayr and Van Kien Nguyen}\thanks{The authors acknowledge support by the Hausdorff Center of Mathematics, University of Bonn}
\date{\today}

\begin{document}

\maketitle

\begin{abstract}
	The problem of recovering a diffusion coefficient $a$ in a second-order elliptic partial differential equation from a corresponding solution $u$ for a given right-hand side $f$ is considered, with particular focus on the case where $f$ is allowed to take both positive and negative values.
	Identifiability of $a$ from $u$ is shown under mild smoothness requirements on $a$, $f$, and on the spatial domain $D$, assuming that either the gradient of $u$ is nonzero almost everywhere, or that $f$ as a distribution does not vanish on any open subset of $D$.
	Further results of this type under essentially minimal regularity conditions are obtained for the case of $D$ being an interval, including detailed information on the continuity properties of the mapping from $u$ to $a$.
\smallskip

\smallskip


\end{abstract}

\section{Introduction}

The problem of identifying diffusion coefficients in second-order elliptic partial differential equations from corresponding solution values arises in many applications. For instance, these  coefficients may play the role of permeability in porous media flows, of electric conductivity in electrostatics, or of thermal conductivity in the stationary heat equation. 

We consider here the classical problem of recovering a scalar diffusion coefficient from distributed measurements of a corresponding solution. Let $D\subset \R^d$ be a bounded domain, $V:= H^1_0(D)$, and let $f \in V'$ and $\varphi \in H^1(D)$ with boundary trace $g := \varphi|_{\partial D}$ be given. For fixed parameters $\lambda,\Lambda>0$ with $\lambda < \Lambda$, we define the set of admissible diffusion coefficients 
\beqq
\mathcal{A}=\{ a\in L_\infty(D): \ 0< \lambda \leq a\leq \Lambda \}
\eeqq
and consider the elliptic problem
\beqq
-\divergence(a\nabla u)=f\qquad \text{on}\ \  D,\qquad u|_{\partial D}= g 
\eeqq
in its weak form given by
\begin{equation}\label{pde}
   \int_D a \nabla u \cdot \nabla v\sdd x = f(v), \; v\in V, \quad\text{and}\quad u - \varphi \in V. 
\end{equation}
For each $a\in \mathcal{A}$ the Lax-Milgram theorem guarantees the existence of a unique solution $u_a\in H^1(D)$ of \eqref{pde}. Here, we are interested in the question under which conditions on $f$, $\varphi$, and $a\in \mathcal{A}$,  the mapping $a\to u_a$ is injective, or in other words, under which conditions we can guarantee that $u_a$ uniquely determines $a$. This property is also referred to as \emph{identifiability} of $a$.

Some restrictions, especially on the source term $f$, are clearly necessary: as soon as $\nabla u_a$ vanishes on an open subset of $D$, $a$ cannot be recovered on that subset via \eqref{pde}. In many existing contributions, this issue is addressed by directly imposing additional assumptions on $\nabla u_a$, or by requiring that $f > 0$. With the latter condition, one can also obtain stronger statements than identifiability, such as stability in the sense of H\"older continuity of $a$ with respect to $u_a$. However, in many situations of interest, the given $f$ may contain both sources and sinks, and thus assuming $f$ to be positive in all of $D$ may be too restrictive.

A second type of restrictions concerns the regularity of the problem data required for showing identifiability.
Except for $d=1$, the existing identifiability results require some additional smoothness of the coefficients $a$ beyond the basic requirement $a \in L_\infty(D)$, even when the restriction $f>0$ is imposed.

We thus aim to understand the identifiability of $a$ in particular without assuming $f$ to have uniform sign, and without further explicit assumptions on $u_a$. 
At the same time, we aim to minimize the additional assumptions on $a$ required for ensuring identifiability. 

\subsection{Main Results and Relation to Previous Work}
The identification of diffusion coefficients from various types of measurements of solutions is a well-studied problem. Besides the case of distributed measurements of $u_a$ that is in our focus here, variants of the celebrated Calder\'on problem of recovering $a$ from the Dirichlet-to-Neumann map have also received significant attention. In the latter setting, one can work with a large set of different boundary data $g$, as typical in applications in electrical impedance tomography.
For the case of identification of $a$ from knowledge of $u_a$ in $D$, the interest is rather in characterizing identifiability given only one set of data $f$, $g$ (as one would expect, for instance, in applications in geophysics).
Provided identifiability holds, one may also consider the continuity properties of the mapping $u_a\mapsto a$.

The most comprehensive treatments have been obtained for $d=1$, where explicit solution formulas are available, see, e.g., \cite{Marcellini,KunischWhite,Boet}. Here \cite{Marcellini,Boet} make the assumption $f > 0$ and work with homogeneous boundary conditions, and \cite{Marcellini,KunischWhite} require the additional regularity $a \in H^1(D)$.

In the higher-dimensional case, one approach that has been used to obtain identifiability is based on the observation that \eqref{pde} defines a transport problem for $a$: for sufficiently regular data, one has
\begin{equation}\label{hyperbolic}
 b \cdot \nabla a + c a = - f,\qquad b :=  \nabla u_a, \; c:=  \Delta u_a .
\end{equation}
This connection is used explicitly in \cite{Richter,RichterNumerical} as well as in \cite{CG}, in both cases assuming $a \in C^1(\bar D)$ and $u \in C^2(\bar D)$, so that \eqref{hyperbolic} can be understood in the classical sense. 
In other works, such as \cite{KohnLowe,Falk,ItoKunisch,Boet}, variational techniques have been used, in each case requiring additional regularity of $a$ and of the further problem data. The main results in \cite{ItoKunisch} and \cite{Boet} are based on estimating $\int_D (a-b)^2/a^2 ( a \abs{\nabla u_a}^2 + u_a f) \sdd x $ in terms of $u_a - u_b$ for coefficients $a,b$, and these arguments rely strongly on $f$ and $u_a, u_b$ having uniform sign. In \cite{Boet}, one has only the relatively weak regularity requirement on the coefficients that $a,b \in H^s(D)$, $s>1/2$.

In our main result Theorem \ref{bvidentifiability}, shown in Section \ref{sec:main-ident}, we require instead $a,b \in BV(D) \cap C^{0,\alpha}(\bar D)$ for some $\alpha >0$, with $D$ a $C^{1,\alpha}$ domain and $f$ allowed to be a distribution with some additional regularity. Moreover, we need that either the boundary data $g$ are constant, or $(a-b)|_{\partial D} = 0$; this type of condition has appeared before, e.g., in \cite{Alessandrini2,Richter,CG}, corresponding to the observation that no Cauchy data are required in the problem \eqref{hyperbolic} when $g$ is constant.
We then obtain identifiability if one of the following two conditions is satisfied: 
\begin{enumerate}[(A)]
	\item  $\nabla u_a$ does not vanish on any set of positive measure; or 
	\item $f$ does not vanish as a distribution on any open subset of $D$, without any further restriction on the sign of $f$.
\end{enumerate}
Note that the condition (A) on $\nabla u_a$ is essentially minimal in view of our above considerations, but depends on the solution $u_a$.
Condition (B), which depends only on the data $f$ is only a sufficient criterion: if $f$ vanishes on some open set, one may or may not have identifiability of $a$. In the particular case that $f$ is a function, (B) holds when $f\neq 0$ a.e.

Although our new technique of proof for this result does not use the interpretation as a hyperbolic problem \eqref{hyperbolic} explicitly, this connection still plays a role, since we extensively rely on recent tools of geometric measure theory, especially from \cite{CTZ}, that were developed for conservation laws with non-smooth data.

We thus obtain identifiability under substantially more general conditions than in \cite{Richter,RichterNumerical,Boet}, where strictly positive $f$ is assumed.
Some further existing results cover rather special cases, for instance 
$f = \delta_{x_0}$ for some $x_0 \in D$ \cite{Alessandrini}; $f=0$ with certain conditions on $g$ \cite{Alessandrini2};
or sources and sinks modelled by inflow and outflow conditions on interior boundaries in $D$ \cite{ChaventKunisch}.
We are not aware of a previous result requiring only a condition on $f$ as in (B) when $f$ is not required to have uniform sign. A result very similar to condition (A) is shown in \cite{CG}, using \eqref{hyperbolic}: there it is shown that if $a \in C^1(\bar D)$, $u \in C^2(\bar D)$, $u|_{\partial D}$ is constant, and the set where $\nabla u$ vanishes has empty interior, then $a$ is identifiable. Earlier similar results were also obtained in \cite{Kunisch} 
and \cite{Falk} under stronger assumptions.

Concerning continuity properties of the mapping $u_a \mapsto a$, most existing results rely on uniform positivity of $f$. Such estimates are of the form
\begin{equation}\label{stab-general}
     \norm{ a - b }_{L_p(D)} \leq C \norm{ u_a - u_b }_V^\gamma 
\end{equation}
for some $C>0$, $\gamma \in (0,1]$, and $1\leq p < \infty$.
Some partial results allowing for more general $f$ have been obtained using assumptions on the Helmholtz decomposition of $\nabla u$ and with finite-dimensional sets of coefficients $a$ in \cite{ChaventKunisch}. The techniques that we develop in Section \ref{sec:main} for showing identifiability do not lend themselves to proving such stability estimates, as explained further in Section \ref{sec:main-stability}. 

In Section \ref{sec:1d}, we demonstrate that for $f$ having non-uniform sign, the exponents $\gamma$ in such stability estimates in general depend strongly on the particular $f$ under consideration.
To this end, we turn to a study of such stability properties in the case $d=1$. In this simpler setting, we obtain results that provide a detailed characterization of H\"older continuity properties \eqref{stab-general}, illustrating the substantial complications that arise when dropping the positivity requirement on $f$. In this case, we also obtain an identifiability result analogous to Theorem \ref{bvidentifiability} that only assumes $a,b\in \cA$ without further regularity. Here we require $f\in L_1(D)$ with $f\neq 0$ almost everywhere in $D$.

\subsection{Piecewise Constant Coefficients}

To illustrate our assumptions on $f$, we consider as a first example a simplified problem with diffusion coefficients varying in a finite-dimensional set of piecewise constant functions.
Let $\{ D_j \colon j=1,\ldots, n\}$ be a partition of $D$ where $D_j, \ j=1,\ldots,n$ are Lipschitz domains  and let
\begin{equation}\label{pwconst}
\cA_n := \biggl\{    \sum_{j=1}^n a_j \Chi_{D_j} \colon  \lambda \leq a_j \leq \Lambda \biggr\}.
\end{equation}
We obtain the following result, inspired by \cite[Theorem~5.2]{Boet} which uses similar assumptions with $f>0$. 
\begin{theorem}\label{pw}
	Let $a,b \in \cA_n$ and $f\in H^{-1}(D)$. Then we have
	\be \label{k-04}
	\abs{ a_i - b_i } \norm{f}_{H^{-1}(D_i)} 
	\leq \Lambda^2 \norm{ \nabla (u_a - u_b )}_{L^2(D_i)}\,.
	\ee
	In particular, if  $\min\{\|f\|_{H^{-1}(D_i)} \colon i=1,\ldots,n\}>0$ then $u_a=u_b$ implies $a=b$\,.
\end{theorem}
 \begin{proof} 
From the variational form, for all $v\in H_0^1(D_i)$ we have
\beqq
\int_{D_i} a\nabla u_a \cdot \nabla (b v) \sdd x= f(bv)
\qquad \text{and}\qquad
\int_{D_i} b\nabla u_b \cdot \nabla (a v) \sdd x= f(av)\,.
\eeqq 
This leads to 
\begin{equation*}
(a_i-b_i)\, f(v) = a_ib_i \int_{D_i}\nabla (u_a-u_b) \cdot \nabla v \sdd x
 \leq \Lambda^2  \norm{ \nabla (u_a - u_b )}_{L^2(D_i)} \| \nabla v\|_{L_2(D_i)}\,.
\end{equation*}
Since this holds for all $v\in H_0^1(D_i)$, we obtain \eqref{k-04}.
\end{proof}

In the case of diffusion coefficients that are piecewise constant on a fixed partition of $D$, we thus have identifiability as soon as $f$ does not vanish in the sense of distributions (that is, as an element of $H^{-1}$) on any of the subdomains. We next come to our main result, where we obtain an analogous condition on $f$.

\section{Coefficients in $BV\cap C^{0,\alpha}$}\label{sec:main}

 In this section we consider the equation with coefficients $a$ in $BV(D)\cap C^{0,\alpha}(\bar{D})$. Recall that for an open subset $\Omega$ of $\R^d$, a non-negative integer $k$, and $0<\alpha\leq 1$, the \emph{H\"older space} $C^{k,\alpha}(\bar{\Omega})$ is defined as the collection of all $u\in C^k(\bar{\Omega})$ such that
 \beqq
 \|u\|_{C^{k,\alpha}(\bar{\Omega})} : = \| u\|_{C^k(\bar{\Omega})} + \sup_{|\beta|=k}\sup_{x,y\in \Omega, x\not =y} \frac{D^\beta u(x)-D^\beta u(y)}{|x-y|^\alpha}<\infty\,.
 \eeqq
  A function $f\in L_1(\Omega)$ has \emph{bounded variation} in $\Omega$, and we write $f\in BV(\Omega)$, if
 \beqq
\|Df\|(\Omega):= \sup\bigg\{ \int_\Omega f\divergence \varphi \sdd x \ \colon\ \varphi \in C_0^1(\Omega,\R^d),\ \|\varphi\|_{L_\infty(\Omega)}\leq 1 \bigg\}<\infty\,.
 \eeqq

\subsection{Identifiability} \label{sec:main-ident}

Our main result, for \eqref{pde} with $d>1$, is the following.
 
 \begin{theorem}\label{bvidentifiability} Let $0<\alpha\leq 1$, and let the following assumptions hold:
 \begin{enumerate}[{\rm(i)}]
 	\item $D$ is a $C^{1,\alpha} $ domain,
 	\item $a,b\in \cA \cap BV(D) \cap C^{0,\alpha}(\bar{D})$,
 	\item we have either $\varphi \in c + V$ for some $c \in \R$ (so that $u_a, u_b \in c+ V$), or $a|_{\partial D} = b|_{\partial D}$,
 	\item $f = f_0 + \divergence F$ where $f_0 \in L_\infty(D)$, $F \in C^{0,\alpha}(\bar D,\R^d)$.
\end{enumerate}
Assume that one of the following two conditions holds:
\begin{enumerate}[{\rm(A)}]
\item $\nabla u_a \neq 0$ a.e.\ in $D$; or
 	\item $f$ does not vanish in the sense of distributions on any open subset of $D$.
 \end{enumerate}
 Then $u_a=u_b$ implies $a=b$.   
 \end{theorem}

Note that the condition (B) in Theorem \ref{bvidentifiability} on $f$ is similar to the requirement for identifiability on $f$ in the piecewise constant case of Theorem \ref{pw}, where we only need that $f$ does not vanish on any subdomain in the partition. In the case that $f$ can be represented as a function, condition (B) in Theorem \ref{bvidentifiability} reduces to $f \in L_\infty(D)$ and $\esssup_{U} \abs{f} > 0$ for any open $U \subset D$, which is implied by $f\neq 0$ a.e.
Before turning to the proof, we give some further remarks on our assumptions. 
 
 \begin{remark} Theorem \ref{bvidentifiability} applies in particular to Lipschitz coefficients $a,b$, but our assumption (ii) is strictly weaker, since $C^{0,1}(\bar D) \subset BV({D})\cap C^{0,\alpha}(\bar{D})$ as a proper subset when $\alpha < 1$.  We also have $W^{1,p}({D})\subset BV({D})\cap C^{0,\alpha}(\bar{D})$ with $d<p$ and $\alpha$ small enough.  \end{remark}
 
\begin{remark}\label{re-k01} 
Under the stated conditions, identifiability generally does not hold if instead of (ii) we only require $a,b\in\mathcal{A}$ without any further continuity assumptions. This is the case even when $f\in L_\infty(D)$, as can be seen from the following example: Denote by $S$ the Smith--Volterra--Cantor set, which is a subset of $[0,1]$ that is nowhere dense, but satisfies $|S|=\frac 12$.
Using $S$, one can easily construct a modification $w$ of Volterra's function \cite{Volterra} that has in particular the following properties: $w = 0$ on $S$; $w'$ exists on $[0,1]$, $w'=0$ on $S$, and $w'$ does not vanish on any interval in $[0,1]$; and $W(x):= \int_0^x w(t) \sdd t$ satisfies $W(1)=0$.
 Setting 
\beqq
f(x)=-w'(x), \qquad x\in [0,1],
\eeqq	
we thus have $\|f\|_{L_\infty(U)}>0$ in any interval $U$. Since $w = w'=0$ on $S$, if $a=b=1$ on $[0,1]\backslash S$ and $a\not=b$ are chosen arbitrarily on $S$, then $a$ and $b$ produce the same solution $u_a=u_b=W$. Thus the statement of Theorem \ref{bvidentifiability} does not hold if $a,b$ are permitted to be discontinuous.
\end{remark}
 
\begin{remark}
With nontrivial inhomogeneous boundary conditions, some further restrictions (for instance that the diffusion coefficients agree on $\partial D$, as in assumption (iii) of Theorem \ref{bvidentifiability}) are unavoidable to ensure identifiability. This is illustrated by the following example: define $u(x) = -\frac12 (x+\frac12)^2$, $x \in [0,1]$, so that $u(0),u(1) <0$ with $u(0)\neq u(1)$. Then one easily checks that for
\[
	a(x) := 1 + \frac1{x+\frac12}, \quad b(x) := 1,
\]
we have $-(a u')' = -(b u')' = 1$ in $(0,1)$, and thus identifiability does not hold in this case.
In the case $d=2$, conditions on boundary data $g$ that guarantee identifiability when $f=0$ have been obtained in \cite{Alessandrini2}.
\end{remark}
 
Next, we collect several notions and auxiliary results that will be used in the proof of Theorem \ref{bvidentifiability}. 
The conditions in Theorem \ref{bvidentifiability} imply the following regularity properties of $u_a$, as shown in \cite[Theorems 8.33 and 8.34]{GT}. 

 \begin{theorem}\label{re}
 	Let $0<\alpha\leq 1$ and let $D$ be a $C^{1,\alpha} $ domain. Assume   $a\in C^{0,\alpha}(\bar{D}) $,
 	$\varphi \in C^{1,\alpha}(\bar{D})$, 
 	and $f = f_0 + \divergence F$ where $f_0 \in L_\infty(D)$, $F \in C^{0,\alpha}(\bar D,\R^d)$. Then the solution $u_a$ of \eqref{pde} belongs to $C^{1,\alpha}(\bar{D})$ and one has
 	\beqq
 	\|u_a\|_{C^{1,\alpha}(\bar{D})} \leq C\bigl(\norm{\varphi}_{C^{1,\alpha}(\bar{D})} + \norm{ f_0}_{L_\infty(D)}  + \norm{F}_{C^{0,\alpha}(\bar{D})}  \bigr) \,.
 	\eeqq
 	The constant $C$ depends on $d$, $\lambda$, $\Lambda$, $\| a\|_{C^{0,\alpha}(\bar{D}) }$, and $D$\,.
 \end{theorem}
 
For establishing the connection to the conditions on $f$ in Theorem \ref{bvidentifiability}, the following simple auxiliary result will be instrumental.
 
 \begin{lemma}\label{subsetzerolemma}
Let $h$ be a bounded and continuous function on $D$ and let $A = \{ h>0\}  \subset D$ have positive measure. Then
 	\begin{equation}\label{subsetzeroint}
 	\int_{A} h |\nabla u_a|^2 \sdd x=0
 	\end{equation}
 	implies that $f$ vanishes in the sense of distributions on an open subset of $A$.
 \end{lemma}
 
\begin{proof}
 	The condition \eqref{subsetzeroint} implies $h \nabla u_a =0$ a.e.\ in $A$. Let $x_0\in A$. Since $h$ is continuous, there exist a ball  $B(x_0,\delta)\subset 
 	A$ centered at $x_0$ with radius $\delta$ such that $h>0$ in $B(x_0,\delta)$, and thus $\nabla u_a(x) = 0$ a.e.\ in $B(x_0,\delta)$. Hence from 
\beqq
f(v) = \int_{B(x_0,\delta)} a \nabla u_a \cdot \nabla v \sdd x =0  
\eeqq
 for all $v\in C_0^{\infty}(B(x_0,\delta))$ we conclude that 	
 	$f=0$ in $B(x_0,\delta)$. 
\end{proof}
 
 The further argument is based on the following observation: with $h = a - b$, for sufficiently regular problem data and any subdomain $A\subset D$,
\begin{equation}\label{formal}
   \int_A \divergence(h u_a \nabla u_a)\sdd x = \int_{\partial A} h u_a (\nabla u_a \cdot\nu) \sdd \mathcal{H}^{d-1},
\end{equation}
and since $\divergence (h \nabla u_a) = 0$ a.e.,
\[
     \int_A \divergence(h u_a \nabla u_a)\sdd x = \int_A h \abs{\nabla u_a}^2 \sdd x.
\]
Now if we could choose $A = \{ h > 0 \}$, then by our assumptions on the boundary data and since $h = 0$ on $\partial A \setminus \partial D$, the right hand side in \eqref{formal} would vanish and we would immediately arrive at
$\int_{\{ h > 0 \}} h \abs{\nabla u_a}^2 \sdd x = \int_{\{ h < 0 \}} h \abs{\nabla u_a}^2 \sdd x= 0$,
and hence at the conclusion of Theorem \ref{bvidentifiability}. However, with the present regularity of $h$ and $u_a$, these steps cannot be carried out directly. In particular, the set $\{ h>0\}$ need not be of finite perimeter, which would be a minimum requirement for justifying a Gauss-Green identity as in \eqref{formal}. For carrying out the strategy outlined above, we thus require results that cover the present low-regularity setting, and in particular we introduce an additional approximation by certain sets $\{ h>t\}$ of finite perimeter with $t \downarrow 0$.

 For a domain $\Omega\subset \R^d$ and a vector field $F\in L_\infty(\Omega,\R^d)$, we set
 \beqq
 \|\divergence F\|(\Omega) =\sup \bigg\{\int_\Omega F\cdot \nabla \phi \sdd x\ \colon \ \phi \in C_0^1(\Omega),\ \|\phi\|_{L_\infty(\Omega)} \leq 1 \bigg\}\,.
 \eeqq
We say that $F$ is a divergence-measure field over $\Omega$ if
\beqq
F\in L_\infty(\Omega,\R^d) \qquad\text{and}\qquad \|\divergence F\|(\Omega)<\infty\,.
\eeqq
 We require the following product rule proved in \cite{CF}.
 \begin{lemma}\label{prod}
 	Let $g$ be a Lipschitz continuous function over any compact set in $\R^d$ and let $F$ be a divergence-measure field. Then $gF$ is a divergence-measure field and
 	\beqq
 	\divergence(gF) = g\divergence F + F\cdot \nabla g.
 	\eeqq
 \end{lemma}

 A set $E\subset \R^d$ is called a set of \emph{finite perimeter} in $\Omega$ if its characteristic
 function $\chi_E$ is a $BV$ function in $\Omega$. We write $P(E;\Omega):= \norm{D\Chi_E}(\Omega)$, where $D\Chi_E$ is the Radon measure defined by the distributional gradient and $\norm{D\Chi_E}$ is the corresponding total variation measure, and we set $P(E) := P(E;\R^d)$. We say that $E$ is of finite perimeter if $P(E) < \infty$. 
The following coarea formula for $BV$ functions relates the $BV$-norm to perimeters of level sets.
 \begin{prop}[Coarea formula] \label{Coa}
 	Let $u\in BV(\Omega)$, and denote $E_t:=\{ x\in \Omega\colon  u(x)>t \}$ for $t\in \R$. Then
 	\beqq
 	\|Du\|(\Omega) = \int_{\R}  P(E_t;\Omega) \sdd t\,.
 	\eeqq
 \end{prop}
 
 The \emph{reduced boundary} $\partial^* E$ is the set of $x \in \R^d$ such that $\norm{D\Chi_E}(B(x,r)) >0$ for all $r>0$ and
\beqq
   \nu_E(x) := -\lim_{r\downarrow 0} \frac{D\Chi_E (B(x,r))}{\norm{D\Chi_E}(B(x,r))}
\eeqq
 exists with $\abs{\nu_E(x)} = 1$.
 Since $P(E)=\mathcal{H}^{d-1}(\partial^*E)$, the set $E$ is of finite perimeter if and only if $\mathcal{H}^{d-1}(\partial^*E) <\infty$, where $\mathcal{H}^{d-1}(\partial^*E)$ denotes the $(d-1)$-dimensional Hausdorff measure of $\partial^*E$ in $\R^d$.  For such sets, the following generalized Gauss-Green theorem holds, see \cite[Theorems 5.2 and 7.2]{CTZ}.
 
 \begin{theorem}\label{G-F} Let $E$ be a set of finite perimeter. If $F$ is a continuous and bounded divergence-measure field on $E$,  then we have
 	\beqq
 	\int_E \divergence F \sdd x =\int_{\partial^*E} F(y)\cdot \nu_{E}(y)\sdd \mathcal{H}^{d-1}(y)\,.
 	\eeqq
 \end{theorem}
 
 We are now in a position to prove Theorem \ref{bvidentifiability}.
 
 \begin{proof}[Proof of Theorem \ref{bvidentifiability}] 
 Let $u_a=u_b$ and $h = a - b$. From the variational form 
we have
\beqq
\int_{D}h \nabla u_a \cdot \nabla v \sdd x =0 \qquad\text{for all $v\in H_0^1(D)$},
\eeqq
or in other words,
 	\beqq
 	\divergence(h\nabla u_a) =0 \qquad \text{a.e.\ in $D$}
 	\eeqq
with the divergence understood in the weak sense. From our assumption we conclude that $u_a\in C^{1,\alpha}(\bar{D})$ by Theorem \ref{re}. Thus $(h\nabla u_a)u_a$ is a continuous and bounded divergence-measure field. 

We extend $h$ from $D$ to a $BV$ function on $\R^d$, still denoted by $h$, with compact support.  Applying Proposition \ref{Coa} we obtain 
 \beqq
 \begin{split} 
 	\int_{0}^{1} P(E_t) \sdd t  = 	\int_{0}^{1} P(E_t,\R^d) \sdd t  
 	& \leq 	 \|Dh\|(\R^d)  <\infty\,,
 \end{split}
 \eeqq
 where $E_t=\{x\in \R^d\colon\ h(x)>t\}$. From this we infer that there exists a decreasing sequence $t_n\to 0$ as $n\to \infty$ such that 
 \be \label{k-03}
P(E_{t_n})= \mathcal{H}^{d-1}(\partial^*\{h> t_n \})  \leq \frac{1}{t_n \abs{\ln t_n}}\,.
 \ee 
 
 For $\alpha \in [0,1]$ and any Lebesgue-measurable set $E\subset \R^d$ we denote 
 \beqq
 E^\alpha :=\bigg\{ y\in \R^d: \lim\limits_{r\to 0} \frac{|E\cap B(y,r)|}{|B(y,r)|} =\alpha \bigg\}\,.
 \eeqq 
 If $E,F \subset \R^d$ are sets of finite perimeter, then up to a set of zero  $\mathcal{H}^{d-1}$-measure we have
 \be \label{k-05}
 \partial^*(E \cap F) \subseteq (\partial^* E \cap F^1) \cup (E^1 \cap \partial^* F) \cup ( \partial^* E \cap \partial^* F)\,,
 \ee  
as shown, e.g., in \cite[Theorem 16.3]{maggi}. Now we consider the set $A_{t_n}=D\cap \{ h>t_n\}$. From \eqref{k-05}  and $D$ being a $C^{1,\alpha}$ domain we have
 \be \label{k-06}
 \begin{split} 
 	\partial^* A_{t_n} & \subseteq (\partial^* D \cap \{ h> t_n\}^1) \cup (D^1 \cap \partial^* \{ h>t_n\}) \cup ( \partial^* D \cap \partial^* \{ h>t_n\})\\
 	& = (\partial D \cap \{ h> t_n\}^1) \cup (\bar{D} \cap \partial^* \{ h>t_n\}) \cup ( \partial D \cap \partial^* \{ h>t_n\})
 	\,.
 \end{split}
 \ee 
 Consequently, we obtain
 \beqq
 \begin{split} 
 	\mathcal{H}^{d-1}(\partial^* A_{t_n})& \leq 2\mathcal{H}^{d-1}
 	(\partial  D ) + \mathcal{H}^{d-1}(\bar{D}\cap \partial^*\{ h>t_n\} ) \\
 	&	\leq 2\mathcal{H}^{d-1}
 	(\partial  D ) +  \mathcal{H}^{d-1}\big(  \partial^*\{ h>t_n\} \big) \\
 	& \leq 2\mathcal{H}^{d-1}
 	(\partial D) +  \frac{1}{t_n \abs{\ln t_n } }<\infty\,.
 \end{split}
 \eeqq
 This implies that for each $n$, the set $A_{t_n}$ is of finite perimeter. 
 
Thus, Theorem \ref{G-F} can be applied to $A_{t_n}$ (where we abbreviate the corresponding measure-theoretic normal vectors by $\nu$ in what follows) to obtain 
 	\beqq
 	\int_{A_{t_n}}  \divergence(h u_a \nabla u_a)\sdd x =\int_{\partial^*A_{t_n}} h u_a (\nabla u_a\cdot \nu) \sdd \mathcal{H}^{d-1}\,,
 	\eeqq
or equivalently  
 	\beqq
 	\int_{A_{t_n}} \divergence (h \nabla u_a) u_a \sdd x + \int_{A_{t_n}} h|\nabla u_a|^2 \sdd x= \int_{\partial^*A_{t_n}} h u_a \nabla u_a\cdot \nu \sdd \mathcal{H}^{d-1}\,,
 	\eeqq
see Lemma \ref{prod}.	From $\divergence (h \nabla u_a)=0 $ a.e. on $D$ and \eqref{k-06} we get
 	\beqq
 	\int_{A_{t_n}} h |\nabla u_a|^2 \sdd x \leq \int_{\partial^* \{h>t_n\} \cap \bar{D}} | h (\nabla u_a\cdot \nu) u_a | \sdd\mathcal{H}^{d-1} + \int_{ \partial D \cap \bar A_{t_n}} | h (\nabla u_a\cdot \nu) u_a | \sdd \mathcal{H}^{d-1}\,.
 	\eeqq
Note that if $\varphi \in c + V$ for a $c \in \R$, in assumption (iii), we can assume $c=0$ without loss of generality, since this modification leaves $\nabla u_a$ unchanged.
Thus on $\partial D $ we have either $u_a=0$ or $h = 0$, which leads to 
 	\beqq
 		\int_{A_{t_n}} h |\nabla u_a|^2 \sdd x \leq \int_{\partial^* \{h>t_n\}\cap D} | h (\nabla u_a\cdot \nu) u_a |\sdd\mathcal{H}^{d-1}\,.
 	\eeqq
 	
 	Note that $h(x)=t_n$ for $x\in \partial^* \{h>t_n\}\cap D$.  Using \eqref{k-03} and the fact that $u_a\in C^{1,\alpha}(\bar{D})$ we can estimate
 	\beqq
 	\begin{split} 
 		\int_{A_{t_n}} h |\nabla u_a|^2   \sdd x
 		& \leq t_n \| u_a\|_{L_\infty(D)} \|\nabla u_a\|_{L_\infty(D)}  \mathcal{H}^{d-1}( \partial^* \{h>t_n\} \cap {D}) \\
 		& \leq C t_n \mathcal{H}^{d-1}(\partial^* \{h>t_n \})\\
 		& \leq C \frac{t_n}{t_n \abs{\ln t_n}}
 	 =\frac{C}{\abs{\ln t_n}}\,,
 	\end{split}
 	\eeqq
 	with $C>0$ independent of $n$. Let $A_{+}=\{x\in D: h(x)>0 \}$, $A_{-}=\{x\in D: h(x)<0 \}$, which by continuity of $h$ are open sets. For $n\to \infty$, by the dominated convergence theorem we conclude that 
 	\beqq
 	\int_{A_+} h |\nabla u_a|^2 \sdd x = \int_{A_-} h |\nabla u_a|^2 \sdd x =0.
 	\eeqq
 	
 	Under condition (A), we directly obtain $A_+ = A_- = \emptyset$, which completes the proof. Let condition (B) hold.
 	If $\abs{A_+}>0$, then Lemma \ref{subsetzerolemma} implies that $f$ vanishes on an open subset of $A_+$, contradicting our assumption. By continuity of $h$, this implies $A_+=\emptyset$.
 	By the same argument, we obtain $A_{-}= \emptyset$, and thus $h=0$ on $D$ also under condition (B).
 \end{proof}

\subsection{Remarks on Stability}\label{sec:main-stability} 

We now comment on the stability of the dependence of $a$ on $u_a$.
Under the conditions of Theorem \ref{bvidentifiability}, the techniques that we develop in Section \ref{sec:main-ident} can not directly be adapted to proving quantitative H\"older estimates of the type \eqref{stab-general}. 
A natural starting point would be to derive
\beqq
\begin{split} 
 \int_{A_{t_n}} h|\nabla u_a|^2 \sdd x & =  \int_{\partial^*A_{t_n}} h u_a \nabla u_a\cdot \nu \sdd \mathcal{H}^{d-1} -\int_{A_{t_n}} b(\nabla u_a-\nabla u_b)\cdot \nabla u_a\sdd x
 \\
 & \qquad  + \int_{\partial^*A_{t_n}} bu_a(\nabla u_a-\nabla u_b)\cdot \nu  \sdd \mathcal{H}^{d-1}
\end{split}
\eeqq
by the same argument as in the proof of Theorem \ref{bvidentifiability}.
However, now the main difficulty is that the behavior of $\nabla u_a-\nabla u_b$ on the set $\partial^*A_{t_n}$ is not clear. Moreover, to extract a stability estimate one also needs to control $|\nabla u_a|^2$ from below. 

With additional a prior bounds that ensure compactness, we can still deduce the following basic continuity result from Theorem \ref{bvidentifiability}.

\begin{cor} Let the assumptions of Theorem \ref{bvidentifiability} hold and for an $M>0$, let
\beqq
a \in \cA_M := \big\{b\in \mathcal{A}:  \norm{b}_{BV(D)} +  \norm{b}_{C^{0,\alpha}(\bar{D})}\leq M \big\}\,.
\eeqq
If $\{a_n\}$ is a sequence in $\cA_M$ and if $u_{a_n}\to u_a$ in $V$, then $a_n\to a$ in $C^{0,\beta}(\bar{D})$ for any $\beta \in (0,\alpha)$. 
\end{cor}

\begin{proof}
The sequence $\{a_n\}$ is compact in $C^{0,\beta}(\bar{D})$. Hence, there exists a subsequence $\{a_{n_k}\}$ converging to $a^*\in \mathcal{A}\cap  C^{0,\beta}(\bar{D})$ in $C^{0,\beta}(\bar{D})$. As a consequence we obtain $u_{a_{n_k}}\to u_{a^*}$ in $V$. But this implies $u_a=u_{a^*}$. 
Since $C^{0,\beta}(\bar{D})\subset L^1(D)$, we also have convergence of $\{a_{n_k}\}$ in $L^1(D)$, and by lower semicontinuity of the $BV$-norm we have $a^* \in BV(D)$. Now from Theorem \ref{bvidentifiability}, applied to $a,a^* \in BV(D)\cap C^{0,\beta}(\bar D)$, we get $a^*=a$. Since the argument shows that any subsequence of $\{a_n\}$ has a subsequence converging to $a$, we obtain that $a_n$ converges to $a$ in $C^{0,\beta}(\bar{D})$. 
\end{proof}

For uniformly positive $f$, H\"older stability estimates of the type \eqref{stab-general} have been obtained, e.g., in \cite{Boet}.  However, under our present assumptions without sign restrictions on $f$, in general the exponents in such H\"older estimates necessarily depend on further particular properties of $f$. To illustrate these difficulties in obtaining quantitative estimates in this setting, we now turn to the case $d=1$, where we arrive at a characterization of H\"older exponents in terms of properties of $f$.

\section{Stability in the One-Dimensional Case}\label{sec:1d}

In this section we study \eqref{pde} in the case $d=1$ with $V=H_0^1(0,1)$ and $\varphi=0$, for which our first result reads as follows.
  
\begin{theorem} \label{1d-1}
	Let $f\in L_1(0,1)$ with $f\not =0$ a.e.\ and $a \in \cA$.
	If  $ \{a_n\}\subset  \mathcal{A}$ such that $u_{a_n}$ converges to $u_a$ in $H_0^1(0,1)$, then $a_n$ converges to $a$ in $L_p(0,1)$ for $1\leq p<\infty$.
	  In particular, if $a,b\in \mathcal{A}$ then $  
	   u_a=u_b$ implies $a=b$ a.e.
\end{theorem}

For $f\in L_1(0,1)$, we set $F(x):=\int_0^xf(t)\sdd t$, so that $F$ is absolutely continuous on $[0,1]$, and we define
\begin{equation}\label{Fminmax}
  F_{\min} :=\min_{x\in [0,1]}F(x)\qquad \text{and}\qquad F_{\max} := \max_{x\in [0,1]}F(x).
  \end{equation}
The proof of Theorem \ref{1d-1} uses the following simple lemma. 

\begin{lemma} \label{lip-01}
	Let $g$ be an absolutely continuous function on $[0,1]$ and let $\Omega=\{ x\in [0,1] : g(x)=0 \}$ with $|\Omega|>0$. Then $g'(x)=0$ a.e. on $\Omega$. 
\end{lemma}
\begin{proof} For every accumulation point $x\in \Omega$, there exists a sequence $\{x_n\}\subset \Omega$ such that $\frac{g(x_n)-g(x)}{x-x_n}=0.$ Since $g(x)$ is absolutely continuous, $g'(x)$ exists a.e.\ on $(0,1)$. Consequently $g'(x)=0$ a.e.\ on $\Omega$.
\end{proof}

\begin{proof}[Proof of Theorem \ref{1d-1}]
We first prove the stated continuity property. Let
\begin{equation}\label{Cadef}
  C_a := \biggl( \int_0^1 \frac{1}{a}\sdd x \biggr)^{-1} {\int_0^1 \frac{F}{a} \sdd x }, \quad a\in \mathcal{A} . 
\end{equation}
It is not difficult to see that $F_{\min}<C_a<F_{\max}$. By continuity, this implies that $F(x)-C_a$  has a zero in $(0,1)$. Since $u'_{a_n}$ converges to $u_a'$ in $L_2(0,1)$ we deduce that there exists a subsequence $u'_{a_{n_k}}$ of $u'_{a_n}$ converging pointwise almost everywhere to $u'_{a}$.   From the variational form,  we have $(au_a')'=-f$ in the weak sense. Since $f\in L_1(0,1)$ we infer that $au'_a$ is absolutely continuous and
	\be  \label{re-001}
	au'_a =  -   F+C_a \,,
\ee
and as a consequence we have
	\beqq
	au_a'(x) - a_{n_k}u'_{a_{n_k}}(x)=C_{n_k}
	\eeqq 
where $C_{n_k}=C_a-C_{a_{n_k}}$. Observe that the sequence $C_{n_k}$ converges to $0$, for suppose this is not the case, we can take $x\in (0,1)$ such that $u'_a(x)$ is sufficiently close to zero and $u'_{a_{n_k}}(x)$ converges to $u_a'(x)$; this is possible by \eqref{re-001}, since $F(x)-C_a$ has a zero.  From the uniform boundedness of coefficients in $\mathcal{A}$ we then deduce a contradiction, and thus $C_{n_k}\to 0$. 

Suppose that $u_a' = 0$ on a set $\Omega_0$ with $\abs{\Omega_0}>0$. 
This implies that 
	$au_a'=0$ a.e. on $\Omega_0$. Since $au'_a$ is an absolutely continuous function, Lemma \ref{lip-01} yields $(au_a')'=0$ a.e.\ on $\Omega_0$. Moreover, for absolutely continuous functions, the weak derivative equals a.e.\ the pointwise derivative. This implies $f=0$ a.e.\ on $\Omega_0$, contradicting our assumption.

Thus $\Omega := (0,1) \setminus  \Omega_0$ has full measure. Let $x\in \Omega$ and $u'_{a_{n_k}}(x)\to u'_a(x)$, then
	\beqq
	[a(x)-a_{n_k}(x)] = \frac{C_{n_k}- a_{n_k}[u'_a(x)-u'_{a_{n_k}}(x)]}{u_a'(x)} \to 0\qquad \text{as}\qquad k\to \infty\,.
	\eeqq
This implies that $a_{n_k}$ converges to $a$ almost everywhere. Hence, from the dominated convergence theorem we conclude that $a_{n_k}$ converges to $a$ in $L_p(0,1)$ for $1\leq p<\infty$. 

The above argument can also be applied to any subsequence of $\{a_n\}$, showing that any subsequence of $\{a_n\}$ has a subsequence converging to $a$ in $L_p(0,1)$, which implies that $a_n \to a$ in $L_p(0,1)$. 
To deduce identifiability, for given $a,b\in\cA$ with $u_a = u_b$ we now choose $a_n := b$ for all $n$. Then we obtain in particular $a_n \to a$ a.e.\ and thus $a = b$ a.e.
\end{proof}

Without the assumption of uniform positivity of $f$, even when Theorem \ref{bvidentifiability} or \ref{1d-1} apply, further H\"older stability results for $a$ in terms of $u_a$ depend more strongly on the particular form of $f$.  To illustrate this further, we consider the following example: Let $u_0\in C_0^\infty(\frac12,1)$ such that $u_0|_{(\frac12,1)} > 0$, $u_0'(\frac34) =0$, $u_0'(x) \neq 0$ for $x \in (\frac12,1)\setminus \{\frac34\}$, and $u_0''\neq 0$ a.e.\ in $(\frac12, 1)$. Here we write $A\lesssim B$ to denote $A\leq C B$ with $C>0$ independent of the quantities in $A$, $B$, and $A\sim B$ to denote $A\lesssim B \wedge B\lesssim A$. For $\alpha >0$, we define 
\[
u(x) = \sum_{k = 0}^\infty 2^{-\alpha k} u_0(2^k x), \; x \in [0,1], \quad\text{and}\quad u(x) = u(-x),\; x \in [-1,0).
\]
Then clearly, with $V = H^1_0(-1,1)$, for any $\alpha > \frac12$ we have 
\[
\norm{u}_{V} \sim \Bigl(\sum_{k=0}^\infty 2^{-k} (2^k 2^{-\alpha k})^2  \Bigr)^{\frac12} < \infty,
\]
and $f := - u''$ satisfies $f\neq 0$ a.e.\ and $\norm{f}_{V'} = \norm{u}_V$. 
Let $\beta \in \R$ and
\[
h_j :=  2^{\beta j} \Chi_{S_j}, \qquad S_j:=(-2^{-j} , 2^{-j}).
\]
Then on the one hand,
\[
\int_{-1}^1 (1 + h_j ) u' v'\sdd x = \int_{-1}^1 fv \sdd x + f_j(v), \quad f_j(v):=  \int_{-1}^1 h_j u' v' \sdd x , \qquad v \in V;
\]
on the other hand, by the Lax-Milgram lemma we have a unique $u_j\in V$ solving
\[
\int_{-1}^1 (1 + h_j ) u_j' v'\sdd x = \int_{-1}^1 fv  \sdd x ,\quad v \in V,
\]
and satisfying the estimate $\norm{u - u_j}_V \lesssim \norm{f_j}_{V'} = \norm{w}_V$,
with $w\in V$ defined as the solution of the variational problem
\[
\int_{-1}^1 w' v' \sdd x =  2^{\beta j}  \int_{S_j}  u' v' \sdd x, \quad v\in V.
\]
Since the unique solution is given by $w = 2^{\beta j} \Chi_{S_j} u \in V$, we obtain
\begin{equation}\label{ujdiff}
\norm{u - u_j}_V  \lesssim 2^{\beta j} \norm{\Chi_{S_j} u }_V \lesssim 2^{\beta j} \Bigl(\sum_{k>j}^\infty 2^{-k} (2^k 2^{- \alpha k})^2  \Bigr)^{\frac12} \sim 2^{ (\frac12 + \beta - \alpha )j}.
\end{equation}
For the corresponding coefficients $a:=1$ and $a_j := 1+h_j$, we have 
\begin{equation}\label{ajdiff}
\norm{a - a_j}_{L_p(-1,1)} \sim 2^{(\beta - \frac1p) j}
\end{equation}
for all $j$. Moreover, if $\beta \leq 0$, so that the functions $h_j$ remain uniformly bounded, we also obtain $\norm{u - u_j }_V \gtrsim \norm{w}_V$ and hence $\norm{u - u_j}_V \sim 2^{ (\frac12 + \beta - \alpha )j}$.

\begin{remark}
Combining \eqref{ujdiff} and \eqref{ajdiff}, we arrive at the following conclusions.
\begin{enumerate}[{\rm(i)}]
	\item The sequence $a_j$ satisfies $\norm{a_j}_{L_\infty(-1,1)} \leq 2$, and thus  $a_j \in \cA$ with $\lambda=1$, $\Lambda=2$, precisely when $\beta \leq 0$. In the case $\beta = 0$, for any $\alpha >\frac12$, we observe $\norm{u_j - u}_V \to 0$ and $\norm{a_j - a}_{L_p(-1,1)} \to 0$ for any $p < \infty$, but $\norm{u_j - u}_V \to 0$ and $\norm{a_j - a}_{L_\infty(-1,1)} =1$. This shows in particular that the continuity statement in Theorem \ref{1d-1} does not extend to $p=\infty$.
	\item For $\beta > 0$, whenever $\alpha > \frac12 + \frac1p$, there are cases with $\norm{u_j - u}_V \to 0$ and $\norm{a_j - a}_{L_p(-1,1)} \nrightarrow 0$ also for $p < \infty$. The corresponding $a_j$, however, do not remain in a set $\cA$ with uniform upper bound in $L_\infty(-1,1)$. The continuity statement in Theorem \ref{1d-1} thus depends crucially on the restriction to $a_j \in \cA$ with some $\Lambda < \infty$.
	\item For any given $\beta\leq 0$ and for $\alpha > \frac12 + \beta$, we have 
\[
    \norm{a_j - a}_{L_p(-1,1)} \lesssim \norm{u_j - u}_V^{\gamma}, \quad \gamma = \frac{p^{-1} - \beta}{\alpha - \frac12 - \beta}.
\]
Thus for large $\alpha$ (corresponding to rapid decay of $\abs{u(x)}$ and $\abs{f(x)}$ as $\abs{x}\to 0$) one obtains arbitrarily small H\"older exponents $\gamma$.
\end{enumerate}
\end{remark}

The following theorem generalizes \cite[Thm.~6.3]{Boet} to $f$ that do not have uniform sign. There, the authors only consider the case $f\geq c_f>0$ (in the particular instance $f\equiv 1$). We recall the notation \eqref{Fminmax}, and for $\rho>0$ and $M\in (F_{\min},F_{\max})$ we set
\beqq
K_\rho(M) :=\{x\in  (0,1): |F(x) - M| \leq \rho \}\,.
\eeqq
 
\begin{theorem}
\label{ratethm1d}
Let $f\in L_1(0,1)$ and $a,b\in \mathcal{A}$. Assume  there exist $\alpha\geq 0$ and $\beta>0$ such that
	\be\label{al-be}
	C_1 \rho^{\alpha} \leq \inf_{F_{\min}<M<F_{\max}}|K_\rho(M)| \leq \sup_{F_{\min}<M<F_{\max}}|K_\rho(M)| \leq C_2 \rho^\beta\,,
	\ee 
with positive constants $C_1,C_2$ independent of $\rho$.  
Then for $1\leq p \leq 2$ we  have
	\begin{equation}\label{thm1d-1}
	  \|a-b\|_{L_p(0,1)} \leq    C  \|u_a'-u_b'\|^{\max\big\{\frac{2\beta}{(2+\alpha)(2 +\beta)},\frac{p\beta}{(p+\alpha)(p +\beta)}\big\} }_{L_2(0,1)}\,,
	\end{equation}
and for $2<p<\infty$, 
	\begin{equation}\label{thm1d-2}
\|a-b\|_{L_p(0,1)} \leq    C  \|u_a'-u_b'\|^{\frac{4\beta}{(2+\alpha)(2 +\beta)p} }_{L_2(0,1)}\,,
\end{equation}
where in each case, $C$ is a positive constant depending on $\lambda,\ \Lambda $, $p$, and $f$.
\end{theorem}

Observe that in the case $1 \leq p \leq 2$, we have $\frac{p\beta}{(p+\alpha)(p +\beta)} \geq \frac{2\beta}{(2+\alpha)(2 +\beta)} $ when $2p\geq \alpha\beta$.

\begin{proof} 
We first prove that for $1\leq p\leq 2$,
	\begin{equation}\label{thm1d-intermediate}
	\|a-b\|_{L_p(0,1)} \leq    C  \|u_a'-u_b'\|^{\frac{p\beta}{(p+\alpha)(p +\beta) }}_{L_p(0,1)}\,,
	\end{equation}
with a $C>0$ having the same dependencies as in \eqref{thm1d-1}, \eqref{thm1d-2}.

Let $\eta=C_a-C_b$, with $C_a$, $C_b$ defined as in \eqref{Cadef}. Without loss of generality we may assume that $\eta\geq 0$. We will show that
	\be\label{d1-001}
\eta \leq c_0\| u'_a-u'_b\|_{L_p(0,1)}^{\frac{p}{p+\alpha}}\,.
	\ee
If $\eta=0$ then the estimate holds. We now consider the case  $\eta>0$. 	We put   $c=\frac{\lambda}{2(\lambda+\Lambda)}<\frac{1}{2}$.  Recall that $C_a\in (F_{\min},F_{\max})$. 
	For $x\in K_{c\eta}(C_a)$, with $A:=\frac{1}{a}$, $B:=\frac{1}{b}$, we have
	\beqq
	\begin{split}
		|u'_a(x)-u'_b(x)| & = |(F(x)- C_b)B(x)-(F(x)-C_a)A(x)| \\
		& \geq (\eta-c\eta)B(x) -c\eta A(x)   \geq \frac{\eta}{2\Lambda}\,,
	\end{split}
	\eeqq
	which yields
$\eta^p|K_{c\eta}(C_a)| \leq (2\Lambda)^p\|u_a'-u_b'\|_{L_p(0,1)}^p$,
	and hence we obtain from our assumption  
	\beqq
	C_1 c^{\alpha} \eta^{p+\alpha}\leq (2\Lambda)^p   \|u_a'-u_b'\|_{L_p(0,1)}^p. 
	\eeqq
	This proves \eqref{d1-001}. We write
	\beqq
	\begin{split}
		(A(x)-B(x))(F(x)-C_a) & =A(x)(F(x)-C_a)-B(x)(F(x)-C_b)+B(x)(C_a-C_b) \\
		&= -(u'_a(x)-u'_b(x)) + B(x)(C_a-C_b) \,.
	\end{split}
	\eeqq
For $x\in {K}_{\rho}(C_a)^c:=(0,1)\backslash K_\rho(C_a)$ and any $\rho>0$ we have 
	\beqq
	\begin{split}
		\rho |(A(x)-B(x)| & \leq |(A(x)-B(x))(F(x)-C_a)|  \leq  |u_a'(x)-u_b'(x)|
		+  \lambda^{-1}\eta
	\end{split}
	\eeqq
	and therefore
	\beqq
	\|A-B\|_{L_p({K}_{\rho}(C_a)^c)}^p  \leq \frac{2^{p-1}}{\rho^p} \|u_a'-u_b'\|^p_{L_p(0,1)} +  \frac{2^{p-1}\lambda^{-p}}{\rho^p}\eta^p \,.
	\eeqq
	For $x \in K_\rho(C_a)$ we have $|A(x)-B(x)|\leq 2\lambda^{-1}$ which leads to
	\beqq
	\|A-B\|_{L_p(K_{\rho}(C_a))}^p  \leq 2^p\lambda^{-p} |K_\rho(C_a)| \leq 2^pC_2\lambda^{-p}\rho^{\beta}\,.
	\eeqq
	Consequently we find
	\beqq
	\begin{split} 
		\|A-B\|_{L_p(0,1)}^p & \leq  \frac{2^{p-1}}{\rho^p} \|u_a'-u_b'\|^p_{L_p(0,1)} +  \frac{2^{p-1}\lambda^{-p}}{\rho^p}\eta^p +2^pC_2\lambda^{-p}\rho^\beta\,. 
	\end{split}
	\eeqq
	Inserting \eqref{d1-001} into this estimate we obtain
	\beqq
	\begin{split} 
		\|A-B\|_{L_p(0,1)}^p  
		&\leq \frac{2^{p-1}}{\rho^p} \|u_a'-u_b'\|^p_{L_p(0,1)} +  \frac{2^{p-1}c_0^p }{\rho^p\lambda^p} \|u_a'-u_b'\|^{\frac{p^2}{p+\alpha}}_{L_p(0,1)} +2^pC_2\lambda^{-p}\rho^\beta 
		\\
		& \leq \frac{C}{\rho^p}\|u_a'-u_b'\|^{\frac{p^2}{p+\alpha}}_{L_p(0,1)} +2^pC_2\lambda^{-p}\rho^\beta\,.  
	\end{split}
	\eeqq
	Now we see that if $u_a=u_b$ then $a=b$ a.e., since we can choose $\rho>0$ arbitrarily small. If $\|u_a'-u_b'\|_{L_p(0,1)}>0$ we put 
	\[
	  \rho= \|u_a'-u_b'\|_{L_p(0,1)}^{p^2/(p+\alpha)(p+\beta)}
	  \]
	  to obtain \eqref{thm1d-intermediate}.
	By the embedding $L_2(0,1) \subset L_p(0,1)$ for $1\leq p \leq 2$, we obtain \eqref{thm1d-1}. 
The estimate \eqref{thm1d-2} follows by interpolation between the case $p=2$ and the bound $\norm{a-b}_{L_\infty(0,1)} \leq 2 \Lambda$.
\end{proof}

\begin{remark}
Let us comment on the condition \eqref{al-be} on the primitive of $f$, which determines the H\"older exponents in the stability estimate \eqref{thm1d-1} for $a$, $b$ . Let $\{ D_j \colon j=1,\ldots, n\}$ be a partition of $D$ and $f=\sum_{j=1}^n a_j \Chi_{D_j}$ with $a_j \neq 0$. It is obvious in this case that $\alpha=\beta=1$. If $f$ vanishes in some open interval then we have $\beta = 0$. Note that one necessarily has $\alpha\geq \beta$; the condition \eqref{al-be} on $\alpha$ essentially puts a limit on the growth of $F$, whereas $\beta$ quantifies regions where $F$ is flat, and both lead to a restriction on the variability of $f$.
	\end{remark}

\section{Conclusions and Open Problems}

We have shown that identifiability of diffusion coefficients $a$ from solutions $u_a$ is ensured under weaker sufficient conditions on source terms $f$ than previously considered in the literature, in particular without a uniform positivity requirement on $f$. We have also shown alternative condition on $\nabla u_a$, which has appeared before in very similar form, to yield the same result under substantially weaker regularity requirements. As demonstrated in our additional study of the one-dimensional case, for the wider class of $f$ considered here, H\"older-type stability properties of the mapping $u_a\mapsto a$ in general depend on particular features of $f$.

One question that remains open is under what conditions uniformly positive $a\in L_\infty(D)$ are identifiable without further regularity requirements when $d\geq 2$. As Remark \ref{re-k01} shows, however, this can only hold under stronger conditions on $f$ than we are using in our main result. In the one-dimensional case, by arguments that are specialized to this situation we indeed see that under such slightly stronger conditions on $f$, identifiability holds under minimal assumptions on $a$.

Another question concerns the further characterization of stability in the higher-dimensional case. As our results for $d=1$ show, one may have arbitrarily small H\"older exponents for the class of $f$ that we consider here. Describing subsets of $f$ that lead to more favorable exponents appears to require new techniques, since neither existing methods for uniformly positive $f$ nor our approach for obtaining identifiability can directly be adapted to address this question.

\bibliographystyle{amsplain}
\bibliography{BNidentifiability}

\end{document}